\documentclass[a4paper,10pt]{amsart}

\usepackage{amsmath}
\usepackage{amssymb}
\usepackage{amsthm}
\usepackage{graphicx}
\usepackage{tabularx}
\usepackage{array}
\usepackage{enumerate}
\usepackage{physics}

\usepackage{tikz}
\usetikzlibrary{positioning}
\usetikzlibrary{arrows.meta}

\DeclareMathOperator{\CH}{conv}

\DeclareMathOperator{\SPAN}{span}

\newcommand{\SET}[2]{\qty{#1:~#2}}
\newcommand{\Real}{\mathbb{R}}
\newcommand{\Espace}[1][n]{\mathbb{R}^{#1}}
\newcommand{\origin}{o}
\newcommand{\mline}[2]{\left\langle #1,#2 \right\rangle}

\newcommand{\msegment}[2]{\left[#1,#2\right]}

\newcommand{\msphere}[1][X]{S_{#1}}
\newcommand{\mcircle}[1][X]{\msphere}
\newcommand{\mball}[1][X]{B_{#1}}
\newcommand{\Mspace}{X}
\newcommand{\diam}[1]{\delta(#1)}
\newcommand{\WSH}[1]{\eta\left(#1\right)}
\newcommand{\TSH}[1]{\theta\left(#1\right)}

\newcommand{\CW}{set of constant width}

\newtheorem{thm}{Theorem}
\newtheorem{cor}[thm]{Corollary}
\newtheorem{lem}[thm]{Lemma}
\newtheorem{prop}[thm]{Proposition}

\theoremstyle{definition}

\newtheorem{rem}[thm]{Remark}

\begin{document}
\title{Uniqueness of completions and related topics}

\author{Chan He}
\address{Department of Mathematics, North University of China, Taiyuan 030051,
  China}
\email{hechan@nuc.edu.cn}

\author{Horst Martini}
\address{Faculty of Mathematics, Chemnitz University of Technology, 09107 Chemnitz, Germany}
\email{horst.martini@mathematik.tu-chemnitz.de}

\author{Senlin Wu}
\address{Department of Mathematics, North University of China, Taiyuan 030051,
  China}
\email{wusenlin@nuc.edu.cn}
\thanks{Senlin Wu is supported by the National Natural Science Foundation of China
(Grant numbers 11371114 and 11571085) and is the corresponding author.}


\maketitle{}

\begin{abstract}
  A bounded subset of a normed linear space is said to be (diametrically)
  complete if it cannot be enlarged without increasing the diameter. A complete
  super set of a bounded set $K$ having the same diameter as $K$ is called a
  completion of $K$. In general, a bounded set may have different completions.
  We study normed linear spaces having the property that there exists a
  nontrivial segment with a unique completion. It turns out that this property
  is strictly weaker than the property that each complete set is a ball, and it
  is strictly stronger than the property that each set of constant width is a
  ball. Extensions of this property are also discussed.
\end{abstract}

\section{Introduction}
\label{sec:introduction}

All normed linear spaces
$\Mspace=(\Mspace,\norm{\cdot})$ 
considered throughout this paper are real and
\emph{nontrivial} (i.e., they contain at least two linearly independent
vectors). When such a space $\Mspace$ is finite dimensional, it is called a
\emph{Minkowski space} (in particular, a two-dimensional Minkowski space is
called a (normed or) \emph{Minkowski plane}). Let $\Mspace$ be such a space. 
We denote by $\mball$ and $\msphere$ the \emph{unit ball} and \emph{unit sphere}
of $\Mspace$, respectively, and by $S_{X^*}$ the unit sphere of the \emph{dual
  space} $X^*$ of $\Mspace$. Moreover, for $x\in\Mspace$ and $\gamma>0$, we
denote by $\mball(x,\gamma)$ the \emph{ball centred at $x$ having radius
  $\gamma$}. For subsets $A,B\subseteq\Mspace$ and $\lambda\in\Real$, we introduce
\begin{displaymath}
  A+B=\SET{x+y}{x\in A,~y\in B}\text{\quad and\quad}\lambda A=\SET{\lambda
    x}{x\in A}.
\end{displaymath}
A subset of $\Mspace$ is said to be \emph{nontrivial} if it contains at least
two points. For brevity, subsets of $\Mspace$ considered in the sequel are all
assumed to be nontrivial.
For $a,b\in\Mspace$, the set
\begin{displaymath}
  \msegment{a}{b}=\SET{\lambda a+(1-\lambda)b}{\lambda\in\msegment{0}{1}}
\end{displaymath}
is called the \emph{segment} connecting $a$ and $b$. When $a\neq b$, 
the set
\begin{displaymath}
  \mline{a}{b}=\SET{\lambda a+(1-\lambda)b}{\lambda\in\Real}
\end{displaymath}
is called the \emph{line} passing through $a$ and $b$.
Suppose that $\msegment{a}{b}$ is
nontrivial. Then $\flatfrac{(b-a)}{\norm{b-a}}$ is called a \emph{direction} of
$\msegment{a}{b}$. Clearly, a nontrivial segment has precisely two directions.
\par

Let $K\subset\Mspace$ be bounded. We denote by $\diam{K}$ the \emph{diameter} of
$K$, 
i.e.,
\begin{displaymath}
  \diam{K}=\sup\SET{\norm{x-y}}{x,y\in K}.
\end{displaymath}
If
\begin{displaymath}
  x\not\in K\Rightarrow \diam{K\cup\{x\}}>\diam{K},
\end{displaymath}
then we say that $K$ is \emph{diametrically maximal}, or \emph{diametrically
  complete}, or simply \emph{complete}. From the definition of complete sets it follows that
\begin{enumerate}
\item $\emptyset$ is not complete while a singleton is always complete;
\item a complete set is bounded, closed, and convex;
\item a closed ball in $\Mspace$ is complete;
\item if $K$ is complete and $\lambda\in\Real$, then $\lambda K$ is also
  complete.
\end{enumerate}

Let $K\subset \Mspace$ be bounded, closed, and convex. For each $f\in S_{X^*}$,
the number
\begin{displaymath}
  w_f(K)=\sup f(K)-\inf f(K)
\end{displaymath}
is called the \emph{width of $K$ in the direction of $f$}. Clearly, for each
$f\in S_{X^*}$ we have
\begin{align*}
  w_f(K)&=\sup\SET{f(x)}{x\in K}-\inf\SET{f(x)}{x\in K}\\
        &=\sup\SET{f(x-y)}{x,y\in K}\\
        &=\sup\SET{f(x)}{x\in K-K}.
\end{align*}
Note that (see, e.g., \cite[Proposition 4]{Caspani-Papini2015})
\begin{displaymath}
  \sup\SET{w_f(K)}{f\in\msphere[X^*]}=\diam{K}.
\end{displaymath}
If
\begin{displaymath}
  w_f(K)=\diam{K},~\forall f\in S_{X^*},
\end{displaymath}
then $K$ is called a \emph{\CW}. The concepts of complete sets and sets of
constant width are extensions of the concept of bodies of constant width in
Euclidean spaces. For more information about these concepts in Euclidean and
Minkowski spaces we refer to the surveys \cite{Chakerian-Groemer1983},
\cite{Heil-Martini1993}, and \cite{Martini-Swanepoel2004}.

\bigskip

Let $K\subset\Mspace$ again be bounded. If $K^C$ is a complete set containing
$K$ and $\diam{K^C}=\diam{K}$, then $K^C$ is called a \emph{completion} of $K$.
In general, $K$ has different completions. We denote by $\Gamma(K)$ the set of
completions of $K$. Then $\Gamma(\cdot)$ is a set-valued map, called the
\emph{diametric completion map}. Let $\mathcal{A}$ be a collection of subsets of
$\Mspace$. If the implication
\begin{displaymath}
  K,L\in \mathcal{A},~\lambda\in\msegment{0}{1}\Rightarrow \lambda K+(1-\lambda)L\in\mathcal{A}
\end{displaymath}
holds, then we say that $\mathcal{A}$ is \emph{convex}. If, for each bounded subset $K$
of $\Mspace$, $\Gamma(K)$ is convex, then we say that $\Gamma(\cdot)$ is
\emph{convex-valued}.
\par
\bigskip
The concepts of wide and tight spherical hulls are useful tools to study
completions of bounded sets. More precisely, the sets
\begin{equation}
  \label{eq:eta}
  \WSH{K}:=\bigcap\limits_{x\in K}\mball(x,\diam{K})\quad\text{and}\quad\TSH{K}:=\bigcap\limits_{x\in\eta(K)}\mball(x,\diam{K})
\end{equation}
are called the \emph{wide spherical hull} and the \emph{tight spherical hull} of
$K$, respectively. Clearly,
\begin{displaymath}
  \WSH{K}=\SET{x\in \Mspace}{\norm{x-y}\leq\diam{K},~\forall y\in K}.
\end{displaymath}
For a nontrivial bounded, closed, and convex set $K$ in a normed linear space
$\Mspace$, $\WSH{K}$ is the union of all completions of $K$, and $\TSH{K}$ is
the intersection of all completions of $K$. Moreover, $K$ has a unique
completion if and only if $\diam{\WSH{K}}=\diam{K}$. See, e.g.,
\cite{Baronti-Papini1995} and, for the case of Minkowski spaces,
\cite{Martini-Richter-Spirova2013}.

\section{The existence of segments having a unique completion}
\label{sec:existence-unique-completion}

Recall that the \emph{modulus of convexity} $\delta_\Mspace(\varepsilon)$
($\varepsilon\in[0,2]$) and the \emph{characteristic of convexity}
$\varepsilon_0(\Mspace)$ of $\Mspace$ are defined by 
\begin{gather*}
  \delta_\Mspace(\varepsilon)=\inf\SET{1-\norm{\frac{x+y}{2}}}{x,y\in\mball,~\norm{x-y}\geq\varepsilon},\\
  \varepsilon_0(\Mspace):=\sup\SET{\varepsilon\in[0,2]}{\delta_\Mspace(\varepsilon)=0},
\end{gather*}
respectively 
(see, e.g, \cite[p. 134]{Fuster2001}).
A Banach space $\Mspace$ is said to be \emph{uniformly nonsquare}
if $\varepsilon_0(\Mspace)<2$.

The starting point of this paper is the following result proved by P.L. Papini
in \cite{Papini2014}.
\begin{thm}[cf. Theorem 3.7 in \cite{Papini2014}]\label{thm:Papini}
  If a Banach space $\Mspace$ is uniformly nonsquare, then every segment in
  $\Mspace$ has different completions. In any Banach space whose dimension is at
  least $2$, there are segments having different completions.
\end{thm}

Therefore, it is natural to consider the following property:
\begin{itemize}
\item[($U_1$)] \emph{There exists a segment $S\subset\Mspace$ which has a unique
  completion.}
\end{itemize}

Clearly, if a nontrivial segment $\msegment{a}{b}$ has a unique completion $C$,
then $C$ is the ball centered at $\flatfrac{(a+b)}{2}$ having radius
$\flatfrac{\norm{a-b}}{2}$.

Since the metric induced by a norm is translation invariant, one can easily verify
the following
\begin{lem}
  Let $\Mspace$ be a normed linear space, and $\msegment{a}{b}$ be a nontrivial
  segment having direction $u$ in $\Mspace$. If $\msegment{a}{b}$ has a unique
  completion in $\Mspace$, then each nontrivial segment in $\Mspace$ having
  direction $u$ has a unique completion.
\end{lem}

By this lemma it suffices to consider segments having the form
$\msegment{-u}{u}$, $u\in\msphere$, to study property ($U_1$).

\begin{thm}\label{thm:char-unique-completion-segment}
  Let $u\in\msphere$. Then the following facts are equivalent:
  \begin{enumerate}
  \item $S:=\msegment{-u}{u}$ has a unique completion,
  \item the unit circle $\msphere[L]$ of each two-dimensional subspace $L$ of
    $\Mspace$ containing $u$ is a parallelogram having $u$ as one of its
    vertices.
  \end{enumerate}
\end{thm}
\begin{proof}
  It is clear that
  \begin{displaymath}
    \WSH{S}=\mball(-u,2)\cap \mball(u,2). 
  \end{displaymath}
  Since $\mball$ is a completion of $S$, $S$ has a unique completion if and only
  if $\WSH{S}=\mball$.
  \par
  First suppose that (2) holds. Since $\mball$ is a completion of $S$ and
  $\WSH{S}$ is the union of all completions of $S$, we have
  $\mball\subseteq\WSH{S}$. Pick an arbitrary point $x$ from $\WSH{S}$. Then
  \begin{equation}\label{eq:in-WSH}
    \norm{x-u},\norm{x+u}\leq 2.
  \end{equation}
  If $x\in\SPAN\qty{u}$, then $x\in S\subseteq\mball$. In the following we
  assume that $x$ and $u$ are linearly independent. Let $L=\SPAN\qty{u,x}$ and
  $v\in\msphere[L]$ be a point such that $\qty{\pm u,\pm v}$ is the vertex set
  of the parallelogram $\msphere[L]$. Suppose that $x=\alpha u+\beta v$. From
  \eqref{eq:in-WSH} it follows that
  \begin{gather*}
    \norm{\alpha u+\beta v-u}=\norm{(\alpha-1)u+\beta v}=|\alpha-1|+|\beta|\leq
    2,\\
    \norm{\alpha u+\beta v+u}=\norm{(\alpha+1)u+\beta v}=|\alpha+1|+|\beta|\leq
    2.
  \end{gather*}
  Therefore, $|\alpha|,|\beta|\leq 1$, which shows that
  $x\in\mball[L]\subseteq\mball$. It follows that $\WSH{S}\subseteq\mball$. Thus
  $\WSH{S}=\mball$.
  \par
  Conversely, we assume that $\WSH{S}=\mball$. Let $L$ be an arbitrary
  two-dimensional subspace of $\Mspace$ containing $u$, and $v\in\msphere[L]$ be
  a point such that $\norm{u+v}=\norm{u-v}$ (the existence of $v$ follows from
  the existence property of isoceles orthogonality,
  see e.g., \cite{Alonso-Martini-Wu2012}). To show that $\msphere[L]$ is a
  parallelogram having $u$ and $v$ as vertices, we only need to prove that
  \begin{displaymath}
    \norm{u+v}=\norm{u-v}=2.
  \end{displaymath}
  Suppose the contrary, namely that $\gamma:=2-\norm{u+v}>0$. Then
  $(1+\gamma)v\not\in\mball$ and
  \begin{displaymath}
    \norm{(1+\gamma)v-u}\leq\norm{v-u}+\gamma=2,~\norm{(1+\gamma)v+u}\leq\norm{u+v}+\gamma=2.
  \end{displaymath}
  It follows that $(1+\gamma)v\in\WSH{S}$, a contradiction to the hypothesis
  that $\WSH{S}=\mball$.
\end{proof}

The following corollary, which is a direct consequence of Theorem
\ref{thm:char-unique-completion-segment}, improves the first part of Theorem
\ref{thm:Papini}.

\begin{cor}
  A normed linear space $\Mspace$ has property ($U_1$) if and only if there exists
  $u\in\msphere$ such that the unit circle $\msphere[L]$ of each two-dimensional
  subspace $L$ of $\Mspace$ containing $u$ is a parallelogram having $u$ as one
  of its vertices.
\end{cor}

\section{Further properties related to complete sets}
\label{sec:further-properties}

For our purpose, we consider now the following properties that a normed linear
space $\Mspace$ might have:
\begin{enumerate}
\item[(A)] \emph{every complete set in the space is of constant width;}
\item[(B)] \emph{every complete set in the space is a ball;}
\item[(C)] \emph{every set of constant width in the space is a ball;}
\item[(D)] \emph{the sum of any two complete sets is complete;}
\item[(E)] \emph{the diametric completion map $\Gamma$ is convex-valued.}
\end{enumerate}

Let $\Mspace$ be a real Banach space. Then $\Mspace$ has property (B) if and
only if $\Mspace$ is a $\mathcal{P}_1$ space, i.e., $\Mspace$ is norm
1-complemented in every Banach space containing it, see \cite{Davis1977}. An
earlier result by Nachbin (see \cite{Nachbin1950}) shows that $\Mspace$ is a
$\mathcal{P}_1$ space if and only if, for every collection
$\qty{\mball\qty(x_i,\gamma_i)}_{i\in I}$ of pairwise intersecting balls, we
always have $\bigcap\limits_{i\in I}\mball\qty(x_i,\gamma_i)\neq\emptyset$, or,
equivalently, the Banach space $\Mspace$ is \emph{hyperconvex}. It is shown in
\cite{Aronszajn-Panitchpakdi1956} that the following theorem holds true.

\begin{thm}[cf. Theorem 2 in
  Section 4 of
  \cite{Aronszajn-Panitchpakdi1956}]\label{thm:existence-extreme-point}
  The unit ball of a hyperconvex Banach space has extreme points.
\end{thm}

\begin{thm}
  If a Banach space $\Mspace$ has property (B), then it has also property
  ($U_1$).
\end{thm}
\begin{proof}
  By Theorem \ref{thm:existence-extreme-point}, $\mball$ has an extreme point
  $u\in \msphere$. We show that $S:=\msegment{-u}{u}$ has a unique completion.
  Since $\Mspace$ has property (B), we only need to prove that $\mball$ is the
  only ball having radius $1$ and containing $S$. Otherwise, there exists a
  point $c\neq\origin$ such that $\norm{u+c}=\norm{u-c}=1$. It follows that $u$
  is the midpoint of a nontrivial segment $\msegment{u+c}{u-c}$ contained in
  $\msphere$, a contradiction to the fact that $u$ is an extreme point of
  $\mball$.
\end{proof}

\begin{rem}
  A Banach space $\Mspace$ having property ($U_1$) does not necessarily have
  property (B). Consider the body of revolution $K$ in $\Espace[3]$ obtained by
  rotating a square in $\Espace[2]$ around one of its diagonals. Then $K$ is the
  unit ball of a Banach space $\Mspace_K$ having property ($U_1$). Theorem 3 in
  \cite{Nachbin1950} shows that a three-dimensional Banach space has property
  (B) if and only if it is isometric to $\qty(\Espace[3],\norm{\cdot}_\infty)$
  , where
  \begin{displaymath}
    \norm{(\alpha,\beta,\gamma)}_\infty=\max\qty{|\alpha|,|\beta|,|\gamma|},~\forall (\alpha,\beta,\gamma)\in\Espace[3].
  \end{displaymath}
  Thus, $\Mspace_K$ does not have property (B). This implies that property (B)
  is strictly stronger than property ($U_1$).
\end{rem}

For each bounded closed set $K$, we put
\begin{displaymath}
  \gamma(K,x)=\sup\{\norm{x-a}:~a\in K\},~\forall x\in \Mspace
\end{displaymath}
and
\begin{displaymath}
  \gamma(K)=\inf\SET{\gamma(K,x)}{x\in\Mspace}.
\end{displaymath}
The number $\gamma(K)$ is called the \emph{circumradius} of $K$.

\begin{lem}[Remark 3.2 in \cite{Papini2015}]\label{lem:char-balls-Papini}
  A complete set $K$
  in a Banach space $\Mspace$
 satisfying
  \begin{displaymath}
    \gamma(K)=\frac{\diam{K}}{2}
  \end{displaymath}
  is a ball.
\end{lem}

\begin{lem}\label{lem:char-intersection-balls-general}
  Let $\Mspace$ be a normed linear space, and $\msegment{x}{y}$ be a segment
  having a unique completion. Then, for each $\gamma\geq
  \flatfrac{\norm{x-y}}{2}$, we have
  \begin{displaymath}
    \mball(x,\gamma)\cap \mball(y,\gamma)=\mball\qty(\frac{x+y}{2},\gamma-\frac{\norm{x-y}}{2}).
  \end{displaymath}
\end{lem}
\begin{proof}
  Without loss of generality, assume that $y=-x\in\msphere$. It is clear that
  \begin{displaymath}
    \mball\qty(\frac{x+y}{2},\gamma-\frac{\norm{x-y}}{2})\subseteq\mball(x,\gamma)\cap \mball(y,\gamma).
  \end{displaymath}
  To show the reverse inclusion, let $w$ be an arbitrary point in
  $\mball(x,\gamma)\cap \mball(y,\gamma)$. The case when $w\in\mline{x}{y}$ is
  clear. In the following we assume that $w$ and $x$ are linearly independent.
  Let $L$ be the Minkowski plane spanned by $x$ and $w$. By Theorem
  \ref{thm:char-unique-completion-segment}, $S_L$ is a parallelogram having $x$
  as one of its vertices. By applying a suitable isometry if necessary, we may
  assume that $L=(\Espace[2],\norm{\cdot}_1)$ and $x=(1,0)$
  , where
  \begin{displaymath}
    \norm{(\alpha,\beta)}_1=|\alpha|+|\beta|,~\forall (\alpha,\beta)\in\Espace[2].
  \end{displaymath}
  It follows that
  \begin{align*}
    w\in B_L(x,\gamma)\cap B_L(y,\gamma)&=\SET{(\alpha,\beta)}{\qty|\alpha-1|+\qty|\beta|\leq\gamma\text{~
                                          and~}\qty|\alpha+1|+\qty|\beta|\leq\gamma}\\
                                        &=\SET{(\alpha,\beta)}{\qty|\alpha|+\qty|\beta|\leq\gamma-1}\\
                                        &=B_L\qty(\frac{x+y}{2},\gamma-\frac{\norm{x-y}}{2})\\
                                        &\subseteq\mball\qty(\frac{x+y}{2},\gamma-\frac{\norm{x-y}}{2}),
  \end{align*}
  which completes the proof.
\end{proof}

\begin{thm}
  If a Banach space $\Mspace$ has property ($U_1$), then it has also property (C).
\end{thm}
\begin{proof}
  Suppose that there exists a segment whose direction is $u\in\msphere$ and
  which admits a unique completion. Then each segment parallel to the line
  $\mline{-u}{u}$ has a unique completion.
  \par
  Let $K$ be a set of constant width. We may assume that $\diam{K}=1$. Then $K$
  is complete. To show that $K$ is a ball we only need to show that
  $\gamma(K)=\flatfrac{\diam{K}}{2}$.
  \par
  By the characterization of sets of constant width, $K-K$ contains the interior
  of $\mball$. Therefore, for each $\varepsilon\in(0,1)$ there exist two points
  $x_\varepsilon$ and $y_\varepsilon$ in $K$ such that
  $x_\varepsilon-y_\varepsilon=(1-\varepsilon)u$. Put
  \begin{displaymath}
    c_\varepsilon=\frac{x_\varepsilon+y_\varepsilon}{2}.
  \end{displaymath}
  We have
  \begin{align*}
    K=\WSH{K}&=\bigcap\limits_{w\in K}\mball(w,\diam{K})\\
             &\subseteq \mball(x_\varepsilon,\diam{K})\cap\mball(y_\varepsilon,\diam{K})\\
             &=\mball\qty(c_\varepsilon,\diam{K}-\frac{1-\varepsilon}{2})\\
             &=\mball\qty(c_\varepsilon,\frac{1+\varepsilon}{2}).
  \end{align*}
  It follows that
  \begin{displaymath}
    \gamma(K)\leq\frac{1+\varepsilon}{2},~\forall \varepsilon\in(0,1).
  \end{displaymath}
  Hence
  \begin{displaymath}
    \gamma(K)=\frac{1}{2}=\frac{1}{2}\diam{K}. 
  \end{displaymath}
  By Lemma \ref{lem:char-balls-Papini}, $K$ is a ball.
\end{proof}

\begin{rem}
  Note that, for our purpose here, a compact, convex set $K$ centred at the
  origin is called \emph{reducible} if there is a non-symmetric compact convex
  set $L$ such that $K = L-L$ holds; otherwise $K$ is said to be
  \emph{irreducible}. Based on this notion, Yost \cite{Yost1991} says that the
  unit ball $\mball$ of a Minkowski space $X$ is reducible if and only if
  $\Mspace$ contains sets of constant width which are not balls. For instance,
  regular icosahedra are irreducible, and hence a 3-dimensional Minkowski space
  $\Mspace$ having a regular icosahedron as unit ball satisfies property (C).
  However, there is no segment in this space having a unique completion. Thus
  property (C) is strictly weaker than property ($U_1$). Therefore, property
  ($U_1$) lies strictly between property (B) and property (C).
\end{rem}

\begin{rem}
  Property ($U_1$) does not imply property (A). For example, let
  $\Mspace=(\Espace[3],\norm{\cdot}_{1})$
  , where
  \begin{displaymath}
    \norm{(\alpha,\beta,\gamma)}_1=|\alpha|+|\beta|+|\gamma|,~\forall (\alpha,\beta,\gamma)\in\Espace[3],
  \end{displaymath}
  and
  $K=\CH\{(-1,-1,-1),(1,1,-1),(1,-1,1),(-1,1,1)\}$. Then $\Mspace$ has property
  ($U_1$), $\diam{K}=4$, $K$ is complete, and $K$ is not of constant width.
  Property (A) does not imply property ($U_1$) either. This can be seen by the
  fact that each Minkowski plane has property (A)
  (cf. \cite[p. 171]{Eggleston1965})
 while not every such plane has
  property ($U_1$). We note that, for a Minkowski space, properties (A), (D),
  and (E) are equivalent, cf. \cite[Proposition 1]{Moreno-Schneider2012IJM}.
\end{rem}

\section{Higher dimensions}
\label{sec:higher-dim}

Let $\Mspace$ be a normed linear space containing at least $n$ linearly
independent vectors, and $m$ be a positive integer not greater than $n$. We say that $\Mspace$
\begin{itemize}
\item \emph{has property ($U_m^b$) if there exists an $m$-dimensional equilateral
simplex in $\Mspace$ whose unique completion is a ball,}
\item \emph{has property ($U_m$) if there exists an $m$-dimensional equilateral
  simplex in $\Mspace$ having a unique completion.}
\end{itemize}
Clearly, a space having property ($U_m^b$) has also property ($U_m$),
and the Euclidean plane $\Mspace$ has property ($U_2$) but does not have
property ($U_2^b$).

\begin{prop}
  Let $\Mspace$ be a normed linear space whose dimension is at least $2$. If
  $\Mspace$ has property ($U_1$), then it has also properties ($U_2^b$) and
  ($U_2$).
\end{prop}
\begin{proof}
  Since $\Mspace$ has property ($U_1$), there exists a unit vector $u$ such
  that the unit sphere of each two-dimensional subspace of $\Mspace$ that
  contains $u$ is a parallelogram having $u$ as one of its vertices. Let $L$ be
  a two-dimensional subspace of $\Mspace$ containing $u$, and $v$ be a vertex of
  the parallelogram $\msphere[L]$ adjacent to $u$. Then the triangle $T$ with
  $\qty{-u,v,u}$ as vertex set is an equilateral $2$-dimensional simplex
  whose diameter is $2$. Each completion of $T$ is a completion of
  $\msegment{-u}{u}$. On the other hand, $\msegment{-u}{u}$ has one unique
  completion that is a ball. Thus, the unique completion of $T$ is a ball, which
  shows that $\Mspace$ has properties ($U_2^b$) and ($U_2$).
\end{proof}

\begin{rem}
  In the proof of the proposition above, we used the fact that a closed segment
  having unique completion can be extended to an equilateral triangle whose
  unique completion is a ball. However, when the dimension of the underlying
  space is at least three, then an equilateral triangle, whose unique completion
  is a ball, in general cannot be extended to an equilateral $3$-simplex. For a
  concrete example we use the idea in the proof of Theorem 4 in
  \cite{Petty1971} (a nice and comprehensive survey on equilateral sets in
  normed spaces is \cite{Swanepoel2004}). Put
  \begin{displaymath}
    K=\CH(\SET{(\alpha,\beta,0)}{\alpha,\beta\in\Real,~\alpha^2+\beta^2=1}\cup\qty{(0,0,1),(0,0,-1)}),
  \end{displaymath}
  and $\Mspace$ be the Minkowski space having $K$ as unit ball. Then
  $u=(0,0,1)$, $-u$, and $v=(1,0,0)$ are the vertices of an equilateral triangle
  $T$ whose unique completion is $K$. It is not difficult to verify that
  \begin{displaymath}
    \SET{x\in\Mspace}{\norm{x+u}=\norm{x-u}=2}=\SET{(\alpha,\beta,0)}{\alpha^2+\beta^2=1}.
  \end{displaymath}
  Therefore, $-v$ is the unique point $y$ in $\Mspace$ such that
  \begin{displaymath}
    \norm{y-v}=\norm{y+u}=\norm{y-u}=2.
  \end{displaymath}
  Thus, $T$ cannot be extended to an equilateral $3$-simplex. However, this does
  not rule out the possibility that property ($U_2$) implies property ($U_3$). 
\end{rem}

\bibliographystyle{plain}

\bibliography{xbib}

\end{document}